\documentclass[final, journal]{IEEEtran}

\pdfoutput=1
\usepackage{enumitem}
\usepackage{amsmath, amssymb}
\usepackage[pdftex]{graphicx}
\usepackage{epstopdf}
\usepackage{float}
\usepackage{subcaption} 
\usepackage{booktabs} 
\usepackage{threeparttable} 
\usepackage{tablefootnote} 

\usepackage{amsthm} 
    \newtheorem{lemma}{Lemma}
    \newtheorem{assumption}{Assumption}
    
    \newtheorem{property}{Property}
    \newtheorem{remark}{Remark}
    \newtheorem{theorem}{Theorem}

    \newtheorem {notation}{Notation}
\renewcommand\qedsymbol{$\blacksquare$} 

\usepackage[ruled, lined, ,longend, linesnumbered]{algorithm2e}

\usepackage[hang,flushmargin]{footmisc}
\usepackage{algpseudocode}
\SetKwInOut{Input}{Input}
\SetKwInOut{Declares}{Declares}
\SetKwInOut{Require}{Require}
\SetKwInOut{Prototype}{Prototype}
\SetKw{KwBy}{by}
\SetKw{Kwor}{or}
\SetKw{KwEndFor}{end for}

\usepackage{fancyhdr}

\newcommand{\becauseof}[2][=]{\stackrel{\scriptstyle\mkern-1.5mu#2\mkern-1.5mu}{#1}}

\def\moveEq#1{{}\mkern#1mu} 

\def\vv#1{{ \rm \bf{#1}}} 
\def\Id{{\rm{I}}}
\def\N{ {\rm \,I\!N} } 
\def\R{{\mathbb{R}}} 
\def\Z{{\mathbb{Z}}} 
\newcommand{\Propper}[1][n]{\Gamma_#1}

\def\cX{{\mathcal{X}}} 
\def\cA{{\mathcal{A}}} 

\def\bx{\bar{x}} 
\def\bn{n} 
\def\bm{\bar{m}} 

\def\set#1#2{\{ \; #1 \;:\;#2\;\}} 
\def\fracg#1#2{{\displaystyle{\frac{#1}{#2}}}}  
\def\Sum#1#2{\sum\limits_{#1}^{#2}} 
\def\Prod#1#2{\prod\limits_{#1}^{#2}} 
\def\ceil#1{\left\lceil #1 \right\rceil} 
\def\floor#1{\left\lfloor #1 \right\rfloor} 

\def\dom#1{{\rm dom}(#1)} 

\def\bmat#1{\left[\begin{array}{#1}} 
\def\emat{\end{array}\right]} 

\newcommand {\T}{^{\top}} 

    \newcommand{\blista}{\begin{enumerate}[label=(\textit{\roman{enumi})}]}
\newcommand{\elista}{\end{enumerate}}

\newcommand {\bsis} {\left\{ \begin{array} }
	\newcommand {\esis} {\end{array}\right.}

\begin{document}
\pagestyle{fancy}

\title{\LARGE Restart of accelerated first order methods with linear convergence under a quadratic functional growth condition}

\author{Teodoro~Alamo,~Pablo~Krupa,~Daniel~Limon
\thanks{The authors are with the Department of Systems Engineering and Automation, Universidad de Sevilla, 41092 Seville, Spain (e-mails: \texttt{talamo@us.es}, \texttt{pkrupa@us.es}, \texttt{dlm@us.es}). Corresponding author: Teodoro Alamo.}%
\thanks{This work was supported in part by Grant PID2019-106212RB-C41 funded by MCIN/AEI/10.13039/501100011033, in part by Grant PDC2021-121120-C21 funded by MCIN/AEI/10.13039/501100011033 and by the ``European Union NextGenerationEU/PRTR", in part by Grant P20\_00546 funded by the Junta de Andalucía and the Fondos Europeos para el Desarrollo Regional (FEDER), and in part by Grant Margarita Salas (20122) funded by the Ministerio de Universidades and the European Union (NextGenerationEU).~\hfill}
\thanks{\copyright{}2022 IEEE. Personal use of this material is permitted. Permission from IEEE must be obtained for all other uses, in any current or future media, including reprinting/republishing this material for advertising or promotional purposes, creating new collective works, for resale or redistribution to servers or lists, or reuse of any copyrighted component of this work in other works.~\hfill}%
}

\maketitle
\thispagestyle{fancy}


\begin{abstract}
Accelerated first order methods, also called fast gradient methods, are popular optimization methods in the field of convex optimization.
However, they are prone to suffer from oscillatory behaviour that slows their convergence when medium to high accuracy is desired.
In order to address this, restart schemes have been proposed in the literature, which seek to improve the practical convergence by suppressing the oscillatory behaviour.
This paper presents a restart scheme for accelerated first order methods for which we show linear convergence under the satisfaction of a quadratic functional growth condition, thus encompassing a broad class of non-necessarily strongly convex optimization problems.
Moreover, the worst-case convergence rate is comparable to the one obtained using a (generally non-implementable) optimal fixed-rate restart strategy.
We compare the proposed algorithm with other restart schemes by applying them to a model predictive control case study.
\end{abstract}

\begin{IEEEkeywords}
Convex Optimization, Accelerated First Order Methods, Restart Schemes, Linear Convergence.
\end{IEEEkeywords}

\section{Introduction} \label{sec:introduction}

In the field of convex optimization, first order methods (FOM) are a widespread class of optimization algorithms which only require evaluations of the objective function and its gradient \cite{Nesterov_S_18, Beck_SIAM_17}. Some examples of these methods include: gradient descent \cite{Nesterov_S_18}, ISTA \cite{Beck09} and ADMM \cite{boyd2011distributed}. A subclass of FOM are the \textit{accelerated} first order methods (AFOM), which are characterised by providing a convergence rate \textit{O}$(1/k^2)$ in terms of the objective function value \cite{Nesterov83}.
Some noteworthy examples are: Nesterov's fast gradient method \cite{Nesterov83}, FISTA \cite{Beck09}, accelerated ADMM \cite{Pejcic_ECC_16}, \cite{Zheng_FADMM_17}, fast Douglas-Rachford splitting~\cite{patrinos2014douglas}, and fast AMA (FAMA) \cite[\S 5]{Goldstein_SJIS_2014}.

The use of AFMOs in the field of control is a heavily researched topic, especially in the field of model predictive control (MPC), as evidenced by the following examples:
\cite{Pu_IFAC_14}, which employs FAMA on a condensed MPC optimization problem; \cite{Alamir:13} and \cite{Richter_TAC_2011}, which consider Nesterov's fast gradient method; \cite{Stathopoulos_TAC_2016}, where the infinite horizon constrained LQR problem is solved using an accelerated dual proximal method; \cite{Patrinos14}, which uses the accelerated dual gradient-projection algorithm; \cite{Kogel_ECC_11} and \cite{Nedelcu_JCO_2014}, which use the fast gradient method along with the augmented Lagrangian method; and \cite{Krupa:PLC:2020}, where the FISTA algorithm is employed.
AFOMs are also used in other areas closely related to the field of control, such as to find the solution of Lasso problems \cite{Tao_JO_2016}, which are employed, for instance, in sparse system identification \cite{Brunton_IFAC_2016}.

A drawback of AFOMs is that they often suffer from oscillating behaviour that slows them down \cite{Donoghue:13}.
In order to mitigate this, restart schemes have been proposed in the literature, which have been shown to improve the convergence in a practical setting by suppressing the oscillatory behaviour.
In a restart scheme, the AFOM is stopped when a certain criterion is met and then restarted using the last value provided by the algorithm as the new initial condition.

Several restart schemes have been proposed in the literature, including the following, where there are three main aspects to consider: \textit{(i)} if they are generally implementable in practice, \textit{(ii)} if they guarantee linear convergence, and \textit{(iii)}, if so, under what assumptions.
In \cite{Donoghue:13} the authors propose two simple heuristic schemes (the functional and gradient restarts) that work well in practice but lack, for the most part, linear convergence guarantees.
The results of this article were extended in \cite{Giselsson:14CDC}, but the conditions for linear convergence are still restrictive.
In \cite[\S 5.2.2]{Necoara:18}, the authors propose a restart scheme, that extends the scheme from \cite[\S 5.1]{Nesterov13}, for optimization problems satisfying a \textit{quadratic functional growth} (QFG) condition \cite[Definition 4]{Necoara:18}, which can be viewed as a relaxation of strong convexity.
This scheme, which guarantees linear convergence, restarts the AFOM after a fixed number of iterations, resulting in a very simple implementation.
However, its drawback is that it requires knowledge of parameters of the optimization problem that are generally hard to obtain, such as the parameter that characterizes the QFG condition.
A restart scheme for FISTA that also guarantees linear convergence for optimization problems satisfying the QFG condition is presented in  \cite{Fercoq_JNA_2019}.
The scheme requires an initial estimation of the QFG parameter, thus not requiring its exact value.
However, we find that it is rather sensitive to this initial guess.
Finally, the authors proposed in \cite{AlamoECC:19} and \cite{AlamoCDC:19} two restart schemes with linear convergence for optimization problems satisfying the QFG condition that do not require knowledge of hard to obtain parameters of the optimization problem nor estimations of them, including the QFG parameter.
However, both schemes are specific to FISTA.

This article presents a novel restart scheme for AFOMs that exhibits linear convergence for optimization problems satisfying a QFG condition, thus encompassing a broad class of non-necessarily strongly convex problems.
Furthermore, it does not require hard-to-attain information of the objective function, such as the QFG parameter.
We provide a theoretical upper bound on the number of iterations needed to achieve a desired accuracy and show that the obtained convergence rate is comparable to the one that could be obtained using the restart scheme from \cite[\S 5.2.2]{Necoara:18}, which is optimal for the class of AFOMs and optimization problems under consideration.

This paper extends the preliminary results presented for FISTA in the conference paper \cite{AlamoECC:19} by providing a restart scheme that is applicable to a broad class of AFOM algorithms and with an improved worst-case convergence rate.

The remainder of this article is structured as follows.
In Section \ref{sec:problem} we formally present the class of optimization problem and AFOM algorithms under consideration.
Section \ref{sec:fixed:rate} describes the optimal fixed-rate restart strategy from \cite{Necoara:18} and provides its iteration complexity for our class of AFOMs.
Section \ref{sec:Alg:Optimal} presents the novel implementable restart scheme with linear convergence.
In Section \ref{sec:Results} we compare the proposed scheme with some of the alternatives referenced above by applying them to solve an MPC problem using the FAMA algorithm from \cite{Pu_IFAC_14}.
We draw conclusions in Section \ref{sec:Conclusions}.

\subsubsection*{Notation}
Given a norm $\|\cdot\|$, we denote by $\|\cdot\|_*$ its dual norm: $\|x\|_*\doteq\sup\,\set{x^T z}{\|z\|\leq 1}$. The $\ell_1$-norm is denoted by $\|\cdot\|_1$. 
$\Z_+$ denotes the set of non-negative integers. The set of integer numbers from $i$ to $j$ is denoted with $\Z_i^j$, i.e. ${\Z_i^j \doteq \{i, i+1, \dots, j-1, j\}}$.
Euler's number is denoted by $e$, and the natural logarithm by ${\ln(\cdot)}$.
${\ceil{x}}$ denotes the smallest integer greater than or equal to $x$, and ${\floor{x}}$ the largest integer smaller than or equal to $x$.
The set of proper closed convex functions from $\R^n$ to $(-\infty,\infty]$ is denoted by~$\Propper$.
Given ${f \in \Propper}$, we denote by $\dom{f}$ its effective domain, that is, ${\dom{f}\doteq\set{x\in \R^n}{f(x)<\infty}}$.
Further notation is given in Notations \ref{notation:f} and \ref{notation:alg} in the following section.

\section{Problem statement}\label{sec:problem}

In this paper we are concerned with finding the solution of optimization problems given by
\begin{equation} \label{eq:OP}
    f^* = \min\limits_{x\in \R^n} f(x),
\end{equation}
which we assume are solvable and where ${f\in \Propper}$.

We use the following notation for the optimal set of \eqref{eq:OP}, the projection operation onto it, and the level sets of $f$.

\begin{notation} \label{notation:f}
    Given the solvable problem \eqref{eq:OP}:  
    \blista
        \item The optimal set is denoted by $\Omega_f$. That is, $$\Omega_f=\set{x\in \R^n}{f(x)=f^*}.$$

        \item For every $x\in \R^n$ we denote $\bx$ the closest element to $x$ in the optimal set $\Omega_f$ with respect to norm $\|\cdot\|$, i.e.
            \begin{equation*}
                \bx =\arg \min\limits_{z\in \Omega_f} \|x-z\|.
            \end{equation*}

        \item Given $\rho \in [0,\infty)$ we denote the level set
            \begin{equation*}
                V_f(\rho) = \set{x\in \R^n}{f(x)- f^* \leq \rho}.
            \end{equation*}
    \elista
\end{notation}

It is well known that if ${f\in \Propper}$ is strongly convex, then there exists $\mu>0$ such that
\begin{equation*}
    f(x)-f^* \geq \frac{\mu}{2}\| x-\bx\|^2, \; \forall x\in \dom{f}.
\end{equation*}
This inequality is called the quadratic functional growth (QFG) condition and it is satisfied, at least locally, for a large class of not necessarily strongly convex functions \cite{Necoara:18}, \cite{drusvyatskiy2018error}, \cite{Wang:14}.

Let us consider a fixed point algorithm $\cA$ that can be applied to solve \eqref{eq:OP}, i.e. given a initial point ${x_0 \in \dom{f}}$, algorithm $\cA$ generates a sequence $\{x_k\}$ with $k \geq 0$ such that ${\lim_{k \to \infty} f(x_k) = f^*}$. We use the following notation to refer to the iterates provided by algorithm $\cA$.

\begin{notation} \label{notation:alg}
    Suppose that the fixed point algorithm $\cA$ is applied to solve problem \eqref{eq:OP} using as initial condition $x_0$. Given the integer $k\geq 1$, we denote with $\cA(x_0,k)$ the vector in $\R^n$ corresponding to iteration $k$ of the algorithm. 
\end{notation}

The following assumption characterizes the class of optimization problems and AFOM algorithms we consider.

\begin{assumption} \label{assum:quad:and:convergence}
    We assume that:
    \blista
        \item For every $\rho>0$, $f\in \Propper$ satisfies the QFG condition
\begin{equation*}
    f(x_0) - f^* \geq \frac{\mu_\rho}{2} \| x_0-\bx_0\|^2, \; \forall x_0\in V_f(\rho)
\end{equation*}
for some $\mu_\rho > 0$.
        \item For every $x_0\in \dom{f}$, algorithm $\cA$ satisfies,
\begin{align}
    f(\cA(x_0,1))  &\leq f(x_0) - \frac{1}{2L_f} \|g(x_0)\|_*^2, \label{ineq:improvement} \\
    f(\cA(x_0,k))-f^* &\leq \frac{a_f}{(k+1)^2} \| x_0-\bx_0\|^2, \; \forall k\geq 1, \label{equ:convergence:rate:FOAM}
\end{align}
where $a_f>0$, $L_f>0$, and $g(\cdot)$ is a gradient operator satisfying $g(x)=0\Leftrightarrow x\in \Omega_f$. \label{assum:quad:and:convergence:fast}
        \item We denote $\bn_\rho \doteq \max \left\{\fracg{1}{2},\sqrt{\fracg{2a_f}{\mu_\rho}} \right\}$. \label{assum:quad:and:convergence:n}
    \elista
\end{assumption}

The conditions listed in Assumption \ref{assum:quad:and:convergence} are expressed in very general terms to be able to account for a broad class of optimization problems and AFOMs.
Let us start by providing some insight and examples of the conditions and terms listed in the assumption.
Assumption \ref{assum:quad:and:convergence}.\textit{(i)} establishes a local QFG condition on the objective function, which encompasses a broad class of non-necessarily strongly convex function (see \cite[Fig. 1]{Necoara:18}).
We refer the reader to \cite{Necoara:18}, \cite{drusvyatskiy2018error} and \cite{Wang:14} for examples of functions satisfying this condition, including the case
$f(x) = h(Ex)+c\T x+\Id_\cX(x)$, where $h:\R^m\to \R$, is a smooth strictly convex function, $E\in \R^{m\times n}$ and $\Id_\cX$ is the indicator function of a polyhedral set $\cX$, which encompasses a large family of optimization problems.
Assumption \ref{assum:quad:and:convergence}.\textit{(ii)} is written in terms of the constants $L_f$ and $a_f$, which will depend both on the AFOM $\cA$ being used and on the structure of $f$, and in terms of an operator $g: \R^n\to\R^n$, which plays the role of the \textit{gradient} operator of $\cA$.
For instance, in the FISTA algorithm applied to $f = h + \Psi$, where $h$ is an $L$-smooth differentiable convex function and $\Psi$ is a (possibly non-smooth) convex function, we have that $L_f = L$, $a_f = 2 L$ and $g$ is the composite gradient mapping operator \cite{AlamoECC:19}.
Condition \eqref{equ:convergence:rate:FOAM} is satisfied by most AFOMs \cite{Nesterov_S_18}, \cite{Beck_SIAM_17}, including the ones we list in the introduction, although under varying assumptions.
Condition \eqref{ineq:improvement} is also satisfied by most AFOMs because the first iteration is often the result of the application of a proximal (or composite) gradient mapping operator.
If this is not the case, then it can be easily enforced by taking this operator as the first step of the algorithm.

In conclusion, Assumption \ref{assum:quad:and:convergence} is satisfied by a broad family of AFOMs and optimization problems of interest in the field of control, including the AFOMs listed in the introduction and optimization problems such as QPs, Lasso or those that arise from many MPC formulations.

We now present a property on the iterates of $\cA$ which serves as the basis for the development and convergence analysis of the optimization schemes of the following sections. An equivalent result can be found in \cite[Subsection 5.2.2]{Necoara:18}.

\begin{property} \label{prop:k:convergence:in:n}
Suppose that Assumption \ref{assum:quad:and:convergence} holds. Then, for every $x_0\in V_f(\rho)$, 
\begin{equation} \label{equ:conv:quad:ff}
    f(\cA(x_0,k))-f^* \leq \left(\fracg{\bn_\rho}{k+1}\right)^2(f(x_0)-f^*),\;\forall k\geq 1.
\end{equation}
\end{property}

\begin{proof}

Denote $f_0 \doteq f(x_0)$, $f_k \doteq f(\cA(x_0,k))$, $\forall k\geq1$. Then, 
 \begin{align*}
     f_k -f^* &\leq \frac{a_f}{(k+1)^2}\| x_0-\bx_0\|^2 \leq \frac{2a_f}{\mu_\rho(k+1)^2} (f_0-f^*) \\
     &\leq \frac{\bn_\rho^2}{(k+1)^2} (f_0-f^*). \qedhere
 \end{align*}

\end{proof}

\section{Optimal fixed-rate restart scheme}\label{sec:fixed:rate}

This section describes the optimal fixed restart scheme presented in \cite[\S 5.2.2]{Necoara:18}, in which $\cA$ is restarted each time the iteration counter attains an optimal fixed number of iterations. We analyze, under Assumption \ref{assum:quad:and:convergence}, its  iteration complexity.

A fixed-rate restart scheme takes the recursion 
\begin{equation} \label{eq:optimal:fixed:rate:recursion}
    v_{j+1} = \cA(v_{j},n),\; j\geq 0,
\end{equation}
where $n\geq 1$ is a fixed integer, starting at a given $v_0\in V_f(\rho)$.

Under Assumption \ref{assum:quad:and:convergence}, the sequence $\{f(v_j)\}_{j\geq 0}$ is non increasing and converges monotonically to $f^*$ if $n\geq\bn_{\rho}$ (see Property \ref{prop:k:convergence:in:n}). Given an accuracy parameter $\epsilon>0$, the following property states the number $M$ of restarts required to satisfy $f(v_{M-1})-f(v_{M})\leq \epsilon$, and shows that the bound on the total number of iterations of $\cA$ is minimized if $n$ is chosen equal to $\ceil{e\bn_\rho}$.
See also \cite[\S 5.2.2]{Necoara:18} for a similar result. 

\begin{property}[Optimal fixed-rate restart scheme] \label{prop:fixed:restart}
    Let Assumption \ref{assum:quad:and:convergence} hold. Given $v_0 \in V_f(\rho)$ and an integer $n$ satisfying $n > \bn_\rho$, consider the recursion \eqref{eq:optimal:fixed:rate:recursion}.
Then, given $\epsilon>0$:
\blista
    \item The inequality $f(v_{M-1})-f(v_{M}) \leq \epsilon$ is satisfied for every $M \geq \bar{M}$, where
\begin{equation} \label{eq:bar:M}
    \bar{M}\doteq 1+\frac{1}{2(\ln n- \ln \bn_\rho)} \ln\,\left(1+ \frac{f(v_0)-f^*}{\epsilon} \right).
\end{equation}
    \item If $n=\ceil{e\bn_\rho}$, the total number of iterations of $\cA$ required to attain $f(v_{j-1}) - f(v_j) \leq \epsilon$ is upper bounded by \label{prop:fixed:restart:iterations}
\begin{equation} \label{ineq:upper:bound:NF}
   \bar{N}_F^* \doteq  \ceil{e\bn_\rho} \ceil{1+ \frac{1}{2} \ln\,\left(1+ \frac{f(v_0)-f^*}{\epsilon} \right)}.
\end{equation} 
In this case, we call recursion \eqref{eq:optimal:fixed:rate:recursion} the \textit{optimal fixed-rate restart scheme}.
\elista
\end{property}

\begin{proof} \renewcommand{\qedsymbol}{}
    See Appendix \ref{app:proof:fixed:restart}.
\end{proof}

One of the key properties of the optimal fixed-rate restart scheme is that it recovers the linear optimal convergence rate provided by Nesterov's fast gradient method for strongly convex functions \cite{Necoara:18}, \cite[\S 2.2]{Nesterov04}.
That is, recalling that $\bn_\rho = \max\{1/2, \sqrt{2 a_f/\mu_\rho} \}$ we easily obtain from Property \ref{prop:fixed:restart}.\ref{prop:fixed:restart:iterations} that an $\epsilon$ accurate solution is obtained in
\begin{equation} \label{eq:optimal:convergence:order}
    \mathcal{O} \left( \bn_\rho \ln \left( \frac{f(v_0) - f^*}{\epsilon} \right) \right)
\end{equation}
iterations.
However, we note that this scheme is often non-implementable because the value of $\bn_\rho$ is generally not available.
Nevertheless, \eqref{ineq:upper:bound:NF} and \eqref{eq:optimal:convergence:order} are of interest because they provide the best theoretical convergence rate that can be obtained with a fixed-rate restart strategy.

\section{Proposed restart scheme} \label{sec:Alg:Optimal}

In this section we propose a novel restart scheme that does not require knowledge of $\bn_\rho$ and that attains a convergence rate similar to the one of the optimal fixed-rate restart strategy described in Section \ref{sec:fixed:rate}. 
We start by presenting Algorithm~\ref{Alg:Delayed}, which implements a performance-based exit condition of algorithm $\cA$. Algorithm \ref{Alg:Delayed} will then be used to derive the main result of this article: Algorithm \ref{alg:Opt:Alg:ERA}.

\begin{algorithm}[t]
    \DontPrintSemicolon
	\caption{Performance-based exit condition on $\cA$} \label{Alg:Delayed}
    \Prototype{$[z,m] = \cA_d(r,n)$}
    \Require{$r\in \dom{f}$, $n \in \R$}
	$x_0 \gets r$, $k \gets 0$\; \label{Ad:line:initialization}
	Initialize $\cA$ with $x_0$\;
    \Repeat{$k\geq n$ and $f(x_\ell)-f(x_k)\leq \fracg{1}{3}\left(f(x_0)-f(x_\ell)\right)$ \label{line:Alg:exit:condition}}{ 
		$k \gets k+1$\;
        $x_k \gets \bsis{cl} \cA(x_0,k) & \mbox{if}\; f(\cA(x_0,k))\leq f(x_{k-1}) \vspace*{0.5em} \\ x_{k-1} & \mbox{otherwise} \esis$\; \label{line:Alg:delayed:min}
		$\ell \gets \floor{\frac{k}{2}}$\;
    }
    \KwOut{$z \gets x_k$, $m \gets k$}
\end{algorithm}

\begin{algorithm}[t]
    \DontPrintSemicolon
	\caption{Optimal Algorithm based on $\cA_d$} \label{alg:Opt:Alg:ERA}
    \Prototype{$[z_{out}, j_{out}] = \cA_*(z_0)$}
    \Require{$z_0 \in \dom{f}$, $\epsilon >0$}
    $m_0 \gets 1$, $m_{-1} \gets 1$, $j \gets -1$\; \label{line:initialization}
    \Repeat{$f(z_{j}) - f(z_{j+1}) \leq \epsilon$ \label{line:exit:condition}}{ 
        $j \gets j+1$\;
        $s_j \gets \bsis{cc}  \sqrt{\fracg{f(z_{j-1})-f(z_{j})}{f(z_{j-2})-f(z_j)}} & \mbox{if} \; j\geq 2 \vspace*{0.5em}\\
        0 & \mbox{otherwise} \esis$\; \label{line:sj}
        $n_j \gets \max\{ m_j, 4s_jm_{j-1}\}$\; \label{line:nj}
        $[z_{j+1},m_{j+1}] \gets \cA_d(z_j, n_j)$\; \label{line:call:Ad}
    }
    \KwOut{$z_{out} \gets z_{j+1}$, $j_{out} \gets j$}
\end{algorithm}

Given an initial condition $x_0$ and a scalar $n$, which serves as a lower bound on the number of iterations, Algorithm \ref{Alg:Delayed} generates a sequence $\{x_k\}_{k\geq 0}$ that satisfies (see step \ref{line:Alg:delayed:min})
\begin{equation*}
    f(x_k) = \min \{f(x_{k-1}), f(\cA(x_0,k))\}, \forall k\geq 1.
\end{equation*}
Therefore,
\begin{equation} \label{equ:min:f:cA}
    f(x_k) = \min\limits_{i=0,\ldots,k} f(\cA(x_0,i)).
\end{equation}
The algorithm exits after $k \geq n$ iterations if the following inequality is satisfied (see step \ref{line:Alg:exit:condition}):
\begin{equation} \label{eq:exit:cond:Alg3}
 f(x_\ell) - f(x_k) \leq \frac{1}{3} \left( f(x_0) - f(x_\ell)\right),
\end{equation}
where $\ell=\floor{\frac{k}{2}}$. 
The outputs of the algorithm are $z\in \R^n$ and $m \in \Z$, where $z = x_m$ and $m \geq n$ is the number of iterations required to satisfy the exit condition \eqref{eq:exit:cond:Alg3}.

Intuitively, exit condition \eqref{eq:exit:cond:Alg3} detects a degradation in the performance of the iterations of $\cA$. Notice that at iteration $m$, the reduction corresponding to the last half of the iterations (from $\floor{\frac{m}{2}}$ to $m$) is no larger than one third of the reduction achieved in the first half of the iterations (from $0$ to $\floor{\frac{m}{2}}$). 
The constant $1/3$ was chosen because it minimized the upper bound on the maximum number of iterations that we show below in Theorem \ref{theo:conv:optimal:algorithm}.\ref{theo:conv:optimal:algorithm:iter:total}.

The following property characterizes the number of iterations required to attain the exit condition \eqref{eq:exit:cond:Alg3} of Algorithm \ref{Alg:Delayed}. This result is instrumental to prove the convergence results of Algorithm \ref{alg:Opt:Alg:ERA}.
\vspace*{-0.5em}
\begin{property}\label{property:Alg:Delayed}
Suppose that Assumption \ref{assum:quad:and:convergence} holds. Then, the output $[z,m]$ from the call $[z,m]=\cA_d(r,n)$ of Algorithm \ref{Alg:Delayed} satisfies, for every $r\in V_f(\rho)$:
\blista
\item $f(z)\leq f(r)- \fracg{1}{2L_f} \|g(r)\|_*^2$, \label{property:Alg:Delayed:grad}
\item $f(z)-f^*\leq \left(\fracg{\bn_\rho}{m+1}\right)^2(f(r)-f^*)$,
\item $n\in (0,\ceil{4\bn_\rho}] \implies m \in [n, \ceil{4 \bn_\rho}]$.
    \elista
\end{property}
\vspace*{-1em}
\begin{proof} \renewcommand{\qedsymbol}{}
    See Appendix \ref{app:proof:Alg:Delayed}.
\end{proof}

We now introduce the main contribution of the article: Algorithm \ref{alg:Opt:Alg:ERA}. This algorithm makes successive calls to Algorithm \ref{Alg:Delayed} (see step \ref{line:call:Ad}) using a minimum number of iterations $n_j$ that is determined by the past evolution of the iterates $z_j$ (see steps \ref{line:sj} and \ref{line:nj}). The main properties of the iterates of Algorithm \ref{alg:Opt:Alg:ERA} are given in the following property and theorem.
\vspace*{-0.5em}
\begin{property} \label{prop:optimal:algorithm}
    Suppose that Assumption \ref{assum:quad:and:convergence} holds and consider Algorithm \ref{alg:Opt:Alg:ERA} for a given initial condition $z_0\in V_f(\rho)$ and accuracy parameter $\epsilon>0$. Then:
\blista
    \item Property \ref{property:Alg:Delayed} can be applied to the iterates of Algorithm \ref{alg:Opt:Alg:ERA} (i.e., taking $r \equiv z_j$, $n \equiv n_j$ $z \equiv z_{j+1}$ and $m \equiv m_{j+1}$). \label{prop:optimal:algorithm:delayed}
    \item The sequence $\{m_j\}$ produced is non-decreasing. In particular, \label{prop:optimal:algorithm:mj}
        \begin{equation} \label{ineq:m:n:m}
            m_j \leq n_j \leq m_{j+1}, \; \forall j \in \Z_0^{j_{out}}.
        \end{equation}
    \item The sequence $\{s_j\}$ satisfies $s_j \in (0, 1]$, $\forall j \in \Z_2^{j_{out}}$. \label{prop:optimal:algorithm:sj}
\elista
\end{property}
\vspace*{-1em}
\begin{proof} \renewcommand{\qedsymbol}{}
    See Appendix \ref{app:proof:theorem}.
\end{proof}
\vspace*{-0.5em}
\begin{theorem}\label{theo:conv:optimal:algorithm}
    Suppose that Assumption \ref{assum:quad:and:convergence} holds and consider Algorithm \ref{alg:Opt:Alg:ERA} for a given initial condition $z_0\in V_f(\rho)$ and accuracy parameter $\epsilon>0$. Then:
\blista
    \item The number of calls to $\cA_d$ (step \ref{line:call:Ad}) is bounded. That is, $j_{out}$ is finite. \label{theo:conv:optimal:algorithm:finite}
    \item The number of iterations of $\cA$ at each call of $\cA_d$ (step \ref{line:call:Ad}) is upper bounded by $\ceil{4 \bn_\rho}$. That is, \label{theo:conv:optimal:algorithm:iter:Ad}
        \begin{equation} \label{ineq:bound:on:mplus}
            m_{j+1} \leq \ceil{4 \bn_\rho}, \; \forall j \in \Z_0^{j_{out}}.
        \end{equation}
    \item The total number of iterations of $\cA$ performed by a call to Algorithm \ref{alg:Opt:Alg:ERA}, which we denote by $N_\cA$, is upper-bounded by $N_\cA \leq \bar{N}_\cA$, where \label{theo:conv:optimal:algorithm:iter:total}
    \begin{equation*}
        \bar{N}_\cA \doteq \frac{e \ceil{4 \bn_\rho}}{2}\ceil{5+ \frac{1}{\ln 15 } \ln\left(1+ \frac{f(z_0)-f^*}{\epsilon} \right) }.
    \end{equation*}
\elista
\end{theorem}
\vspace*{-1em}
\begin{proof} \renewcommand{\qedsymbol}{}
    See Appendix \ref{app:proof:theorem}.
\end{proof}

\begin{remark} \label{rem:exit:condition}
From Property \ref{prop:optimal:algorithm}.\ref{prop:optimal:algorithm:delayed}, we have that we can rearrange Property \ref{property:Alg:Delayed}.\ref{property:Alg:Delayed:grad} to read as
\begin{equation*}
    \|g(z_j)\|_*^2 \leq 2L_f(f(z_j)-f(z_{j+1})).
\end{equation*}
Therefore, the exit condition $f(z_j) - f(z_{j+1}) \leq \epsilon$ implies $\| g(z_j) \|_*^2 \leq 2 L_f \epsilon$. Since, as per Assumption \ref{assum:quad:and:convergence}.\ref{assum:quad:and:convergence:fast}, $g(z_j)$ serves to characterize the optimality of $z_j$, we conclude that the exit condition of Algorithm \ref{alg:Opt:Alg:ERA} also serves to characterize the optimality of $z_{j + 1}$.
This means that the exit condition could be replaced by $\| g(z_j) \|_* \leq \tilde \epsilon$, where $\tilde \epsilon > 0$.
In this case, the upper bound on the number of iterations given in Theorem \ref{theo:conv:optimal:algorithm}.\ref{theo:conv:optimal:algorithm:iter:total} would be the same but replacing $\epsilon$ with $\tilde \epsilon/(2 L_f)$.
\end{remark}

Note that Theorem \ref{theo:conv:optimal:algorithm}.\ref{theo:conv:optimal:algorithm:iter:total} shows that the proposed algorithm attains the optimal linear convergence rate of the optimal fixed-rate restart scheme, in the sense that an $\epsilon$ accurate solution is obtained in \eqref{eq:optimal:convergence:order} iterations.
This is an important result, since we recover the optimal convergence rate for our class of optimization problems using AFOMs.

Comparing the upper bound provided in Theorem \ref{theo:conv:optimal:algorithm}.\ref{theo:conv:optimal:algorithm:iter:total} with the upper bound $\bar{N}_F^*$ \eqref{ineq:upper:bound:NF} of the optimal fixed-rate restart scheme presented in Section \ref{sec:fixed:rate}, we have
\begin{equation*}
\frac{\bar{N}_\cA}{ \bar{N}_F^*} =  \frac{e \ceil{4 \bn_\rho}\ceil{5+ \frac{1}{\ln 15 } \ln\left(1+ \fracg{f(z_0)-f^*}{\epsilon} \right) }}{2\ceil{e \bn_\rho}\ceil{1+ \frac{1}{2 } \ln\left(1+ \fracg{f(z_0)-f^*}{\epsilon} \right)}},
\end{equation*}
from where we obtain that
\begin{eqnarray*} \lim\limits_{\epsilon \to 0} \frac{\bar{N}_A}{\bar{N}_F^*} &=& \frac{e\ceil{4\bn_\rho}}{\ceil{e\bn_\rho}\ln 15} \leq  
\frac{e(4\bn_\rho+1)}{e\bn_\rho \ln 15} \\
&=& \frac{4}{\ln 15} + \frac{1}{\bn_\rho \ln 15}\leq \frac{3}{2}\left(1+ \frac{1}{4 \bn_\rho} \right).
\end{eqnarray*}
We conclude that the worst case complexity of (the implementable) Algorithm \ref{alg:Opt:Alg:ERA} is comparable to the (generally) non implementable optimal fixed-rate restart scheme (approximately 50\% more iterations of $\cA$ when $\epsilon$ tends to zero). 

\section{Numerical results} \label{sec:Results}

We compare Algorithm \ref{alg:Opt:Alg:ERA} with other restart schemes of the literature by applying them to the FAMA algorithm for MPC from \cite{Pu_IFAC_14}, where we consider the MPC formulation \cite[Eq. (2)]{Krupa_arXiv_ellipMPC_21_v2} but without its terminal constraint (2f). This MPC formulation can be posed as \cite[Problem 2.1]{Pu_IFAC_14}, which can therefore be solved using the FAMA algorithm \cite[Alg. 1]{Pu_IFAC_14}.
As stated in \cite[\S 3.2]{Pu_IFAC_14}, this algorithm is equivalent to applying FISTA to the dual problem of \cite[Problem 2.1]{Pu_IFAC_14}.
Therefore, the dual objective function value $D(\lambda)$ (see \cite[\S 3.2]{Pu_IFAC_14}) and its gradient $\nabla D(\lambda)$, where $\lambda$ are the dual variables, satisfy Assumption \ref{assum:quad:and:convergence}, with $\lambda_k \equiv x_k$, $D \equiv f$, $\nabla D \equiv g$, $L_f$ is equal to the expression $\rho(C)/\sigma_f$ described in \cite[\S 2]{Pu_IFAC_14} and $a_f = 2 L_f$.

We consider the system described in \cite[\S 3]{Krupa_arXiv_ellipMPC_21_v2}, which consists of three objects connected by springs where external forces can be applied to the two outer-most objects.
We use the exact same setup as in \cite[\S 3]{Krupa_arXiv_ellipMPC_21_v2} with the exception of the cost function matrix $T$, which we take as the solution of the discrete algebraic Ricatti equation, as is often the case in MPC, and the state constraints, which we do not enforce.
We perform the preconditioning procedure described in \cite[\S 4]{Pu_IFAC_14} on the resulting MPC optimization problem.

We solve the MPC's optimization problem for $1000$ randomly generated system states, where the positions of each object is obtained from a uniform distribution on the interval $[0, 4]$ dm and the velocities from a uniform distribution on the interval $[-0.5, 0.5]$ m/s.
We solve each problem using the FAMA algorithm \cite[Alg. 1]{Pu_IFAC_14} with different restart schemes:
\blista
    \item \textbf{\textit{Functional}:} The heuristic restart scheme proposed in \cite{Donoghue:13} that uses restart condition ${D(\lambda_{k+1}) \leq D(\lambda_{k})}$.
    \item \textbf{\textit{Gradient}:} The heuristic restart scheme proposed in \cite{Donoghue:13} that uses restart condition ${\langle \nabla D(\lambda_k), \lambda_k - \lambda_{k+1} \rangle \geq 0}$.
    \item \textbf{\textit{Optimal}:} The restart scheme proposed in \cite[\S 5.2.2]{Necoara:18}, which is typically non-implementable.
        In order to include for comparison with the proposed scheme, we implement it using the method described at the end of the referenced subsection that requires knowing $D(\lambda^*)$, which we compute using FAMA with a very small exit tolerance.
    \item \textbf{\textit{GLCR}:} The restart FISTA algorithm with linear convergence proposed in \cite[Alg. 2]{AlamoECC:19}.
    \item \textbf{\textit{GBRF}:} The gradient-based restart FISTA algorithm proposed in \cite[Alg. 2]{AlamoCDC:19}.
    \item \textbf{\textit{Adaptive}:}  The adaptive scheme for FISTA from \cite{Fercoq_JNA_2019}. We take the initial guess of the QFG parameter as $10^{-5}$.
\elista
We also use the proposed restart scheme (Algorithm \ref{alg:Opt:Alg:ERA}) and the non-restarted FAMA.

Table \ref{tab:iters} shows the average, median, maximum and minimum number of iterations of each scheme for the $1000$ tests.
The results of the table are obtained by terminating FAMA when the iterate $\vv{u}_k$ of the primal problem (see \cite[Alg. 1]{Pu_IFAC_14}) satisfies $\| \vv{u}_k - \vv{u}^*\|_2/\|\vv{u}^*\|_2 \leq 10^{-5}$, where $\| \cdot \|_2$ stands for the standard Euclidean norm and the optimal solution $\vv{u}^*$ of the primal problem is obtained using the \texttt{quadprog} solver from Matlab.
We use this exit condition to provide a fair comparison between the different approaches.
Figure~\ref{fig:evolution} shows the evolution of $\| \vv{u}_k - \vv{u}^*\|_2/\|\vv{u}^*\|_2$ for each restart scheme when taking the system state as the origin.
The results indicate that the restart schemes tend to reduce the number of iterations required to obtain a solution of the MPC's optimization problem when compared with the non-restarted variant, although this is not always the case.

{\renewcommand{\arraystretch}{1.4}%
\begin{table}[t]
    \centering
    \caption{Comparison between restart schemes.} 
    \label{tab:iters}
    \setlength\tabcolsep{5pt} 
        \begin{tabular}{|c|c|c|c|c|}
        \hline
        Scheme & Avg. Iter. & Med. Iter. & Max. Iter. & Min. Iter \\\hline\hline
        No restart & 6262.7 & 5124.5 & 26659 & 48 \\\hline
        Alg. \ref{alg:Opt:Alg:ERA} & 1115.9 & 1080.5 & 2902 & 23 \\\hline
        \textit{GLCR} & 1103.0 & 1070.5 & 2963 & 29 \\\hline
        \textit{GBRF} & 2184.0 & 2132.5 & 5991 & 48 \\\hline
        \textit{Optimal} & 2135.3 & 2095.0 & 5893 & 53 \\\hline
        \textit{Functional} & 845.2 & 801.0 & 4134 & 26 \\\hline
        \textit{Gradient} & 835.5 & 801.5 & 2147 & 26 \\\hline
        \textit{Adaptive} & 1858.1  & 1842.5 & 3827 & 48 \\\hline
        \end{tabular}
\end{table}}

\begin{figure}[t]
    \centering
    \includegraphics[width=\columnwidth]{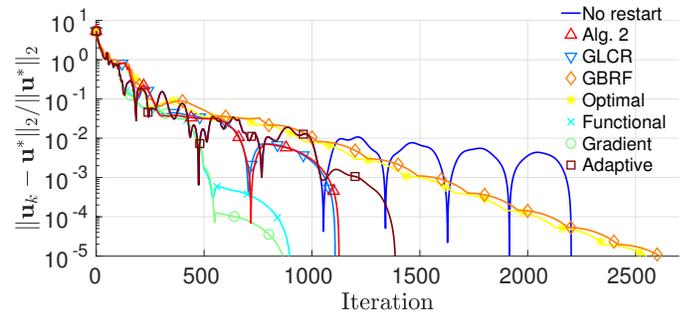}
    \caption{Evolution of distance to optimality for each scheme.}
    \label{fig:evolution}
\end{figure}

Additional numerical results can be found in \cite[\S 3.3]{Krupa_Thesis_21}, where the above restart schemes are applied to solve random QP and Lasso problems with varying condition numbers.

\section{Conclusions} \label{sec:Conclusions}

We propose a restart scheme applicable to a broad class of AFOMs that does not require knowledge of hard-to-obtain parameters of the optimization problem and still retains a linear convergence rate similar to the optimal one for optimization problems satisfying a QFG condition, i.e., its worst case complexity is similar to the best one that can be obtained if the parameters characterizing the convergence of the AFOM were known. 
The scheme is based on a performance-based exit condition that detects the degradation of the convergence of the AFOM and which we show is attained in a finite number of iterations.
The numerical results indicate that the proposed algorithm is comparable, in practical terms, with other restart schemes of the literature.
In our numerical results we find that the application of restart schemes may be desirable in many situations, since the non-restarted variant can require an excessive number or iterations.

\begin{appendix}
 

\subsection{Proof of Property \ref{prop:fixed:restart}} \label{app:proof:fixed:restart}

\noindent Suppose that the integer $M$ is such that the inequality $ f(v_{M-1})-f(v_{M})>\epsilon$ is satisfied.
From Property \ref{prop:k:convergence:in:n} we have
\begin{equation*}
f(v_{j+1}) - f^* \leq  \left(\frac{\bn_\rho}{n}\right)^2 (f(v_{j})-f^*),\; \forall j\geq 0.
\end{equation*}
Using this inequality in a recursive manner we obtain
\begin{align*}
    \epsilon &< f(v_{M-1})-f(v_{M}) \leq f(v_{M-1})-f^* \\
    &\leq \left( \frac{\bn_\rho}{n} \right)^{2(M-1)} (f(v_0)-f^*). 
\end{align*} 
This leads to 
\begin{equation} \label{eq:proof:fixed:restart:M}
    M < 1 + \frac{1}{2(\ln n-\ln \bn_{\rho})} \ln \left(\frac{f(v_0)-f^*}{\epsilon}\right) < \bar{M}.
\end{equation}
Thus, we conclude that if $M$ does not satisfy \eqref{eq:proof:fixed:restart:M}, then $f(v_{M-1})-f(v_{M}) \leq \epsilon$. 
This proves the first claim.

Given $\epsilon>0$, $v_0 \in V_f(\rho)$ and $n > \bn_\rho$, denote $S\geq 0$ the smallest number of restarts required to satisfy the condition $f(v_{S-1})-f(v_{S})\leq \epsilon$. We infer from the first claim of the property that
\begin{equation*}
    S \leq \ceil{1+\frac{1}{2(\ln n- \ln \bn_\rho)} \ln\,\left( 1+\frac{f(v_0)-f^*}{\epsilon} \right)}.
\end{equation*}
Since each restart requires $n$ iterations of $\cA$, we conclude that $N_F(n)$, the total number of iterations of $\cA$, is equal to $nS$. Thus, 
\begin{equation}\label{equ:bound:N:F:n} 
    \moveEq{-10}  N_F(n) {\leq} n \ceil{1{+}\frac{1}{2(\ln n- \ln \bn_\rho)} \ln\,\left( 1{+}\frac{f(v_0)-f^*}{\epsilon} \right)}.
\end{equation}
Simple calculus yields that the value that minimizes
$\frac{n}{\ln n-\ln \bn_\rho}$
is $n^* = e\bn_\rho$. Since $n$ has to be a positive integer, we choose the fixed restart rate given by $n=\ceil{e\bn_\rho}$. Introducing this value in the bound \eqref{equ:bound:N:F:n} we finally obtain
\begin{equation*}
    N_F(\ceil{e\bn_\rho}) \leq  \ceil{e\bn_\rho} \ceil{ 1+ \frac{1}{2} \ln\,\left(1+ \frac{f(v_0)-f^*}{\epsilon} \right)}. \qed
\end{equation*}

\subsection{Proof of Property \ref{property:Alg:Delayed}} \label{app:proof:Alg:Delayed}

\noindent From \eqref{equ:min:f:cA} and Assumption \ref{assum:quad:and:convergence} we have
\begin{align*}
    f(z) &= f(x_m) \becauseof[=]{\eqref{equ:min:f:cA}} \min\limits_{i=0,\ldots,m} f(\cA(x_0,i)) \\
    &\leq f(\cA(x_0,1)) \becauseof[\leq]{\eqref{ineq:improvement}} f(x_0) -\frac{1}{2L_f} \| g(x_0)\|_*^2.
\end{align*}
The first claim now follows directly from $x_0=r$. In view of Property \ref{prop:k:convergence:in:n} we have, for every $k\in\Z_1^m$, 
\begin{equation} \label{ineq:fk:in:Alg:delayed}
\moveEq{-10} f(x_k){-}f^* \becauseof[\leq]{\eqref{equ:min:f:cA}} f(\cA(x_0,k)) {-}f^* \becauseof[\leq]{\eqref{equ:conv:quad:ff}} \Big(\frac{\bn_\rho}{k{+}1}\Big)^2 (f(x_0){-}f^*). 
\end{equation}

The second claim now follows from $x_m=z$ and $x_0=r$. 
The inequality $n\leq m$ is trivially satisfied from step \ref{line:Alg:exit:condition}. Thus, in order to conclude the proof we show that inequality \eqref{eq:exit:cond:Alg3} is satisfied for $\hat{k}=\ceil{4\bn_\rho}$ and $\hat{\ell}=\floor{\frac{\ceil{4\bn_\rho}}{2}}\geq \floor{2\bn_\rho}\geq 1$ (where this last inequality follows from Assumption \ref{assum:quad:and:convergence}.\ref{assum:quad:and:convergence:n}, which states that $\bn_\rho \geq 1/2$). 
\begin{align*}
    f(x_{\hat{\ell}})-f^* &\becauseof[\leq]{\eqref{ineq:fk:in:Alg:delayed}} \left(\frac{\bn_\rho}{\hat{\ell}+1}\right)^2 (f(x_0)-f^*) \\
    & \leq \left(\frac{\bn_\rho}{2\bn_\rho}\right)^2 (f(x_0)-f^*) =\frac{1}{4} (f(x_0)-f^*).
\end{align*}
This implies 
$f(x_{\hat{\ell}}) \leq \frac{1}{4} f(x_0) + \frac{3}{4}f^* \leq \frac{1}{4} f(x_0) + \frac{3}{4}f(x_{\hat{k}}).$
Thus,
\begin{align*}
    f(x_{\hat{\ell}})-f(x_{\hat{k}}) &\leq \frac{1}{4} (f(x_0)-f(x_{\hat{k}})) \\
    &=\frac{1}{4} (f(x_0)-f(x_{\hat{\ell}})) + \frac{1}{4}(f(x_{\hat{\ell}})-f(x_{\hat{k}})).
\end{align*}
We conclude that $f(x_{\hat{\ell}})-f(x_{\hat{k}}) \leq \fracg{1}{3}(f(x_0)-f(x_{\hat{\ell}}))$. \qed

\subsection{Proofs for Algorithm \ref{alg:Opt:Alg:ERA}} \label{app:proof:theorem}

\begin{proof}[Proof of Property \ref{prop:optimal:algorithm}]

Since $z_0 \in \dom{f}$, we have that $z_0 \in V_f(\rho)$ for some $\rho > 0$. Additionally, each $z_{j}$ is obtained from a call to Algorithm \ref{Alg:Delayed} (step \ref{line:call:Ad}). As such, in view of Property \ref{property:Alg:Delayed}.\ref{property:Alg:Delayed:grad}, we have that the iterates $z_j$ satisfy $z_j \in V_f(\rho)$, $\forall j \in \Z_0^{j_{out}}$. Therefore, Property \ref{property:Alg:Delayed} can be applied to each call to $\cA_d$, thus proving claim \ref{prop:optimal:algorithm:delayed}. That is, for every $j \geq 0$, the iterates of Algorithm \ref{alg:Opt:Alg:ERA} satisfy:
\begin{subequations}
\begin{align}
    &f(z_{j+1}) \leq f(z_{j}) -\frac{1}{2L_f}\|g(z_{j})\|^2_*,
    \label{equ:monotocity:f:z} \\
    &f(z_{j+1})-f^* \leq   \left(\frac{\bn_\rho}{m_{j+1}+1}\right)^2(f(z_{j})-f^*), \label{ineq:recursion:improvement}\\ 
    &n_j\in (0,\ceil{4\bn_\rho}] \Rightarrow m_{j+1} \in [n_j,\ceil{4\bn_\rho}]. \label{ineq:m:jplut:nj}
\end{align}
\end{subequations}

Next, due to step \ref{line:nj} we have $m_j \leq n_j$, $j \in \Z_0^{j_{out}}$. Moreover, from \eqref{ineq:m:jplut:nj}, we have that $n_j \leq m_{j+1}$, ${\forall j \in \Z_0^{j_{out}}}$, which proves claim \ref{prop:optimal:algorithm:mj}.

Finally, we prove claim \ref{prop:optimal:algorithm:sj}. From the exit condition (step \ref{line:exit:condition}), we have
\begin{equation}\label{ineq:smaller:than:epsilon}
f(z_{j-1})-f(z_{j}) >\epsilon, \;\forall j\in \Z_1^{j_{out}}.
\end{equation}
Additionally, from \eqref{equ:monotocity:f:z} we have $f(z_{j-2})\geq f(z_{j-1})$, ${\forall j \in \Z_2^{j_{out}}}$. Thus,
\begin{equation*} \label{ineq:cociente}
f(z_{j-2})-f(z_j) \geq f(z_{j-1})-f(z_j) \becauseof[>]{\eqref{ineq:smaller:than:epsilon}} \epsilon>0, \; \forall j\in\Z_2^{j_{out}}.
\end{equation*}
Therefore, from step \ref{line:sj}, taking $j \geq 2$, we have 
\begin{equation*}
0 < s_j =  \sqrt{\fracg{f(z_{j-1})-f(z_{j})}{f(z_{j-2})-f(z_j)}} \leq 1,\;   \forall j\in \Z_2^{j_{out}}. \qedhere
\end{equation*}

\end{proof}

The proof of the following lemma relies upon some technical results on the iterates of Algorithm \ref{alg:Opt:Alg:ERA}, namely Lemmas \ref{lemma:lower:bound:expression} and \ref{lemma:T}, which we include in Appendix \ref{app:technical:lemmas}.

\begin{lemma} \label{lemma:D}
    Consider Algorithm \ref{alg:Opt:Alg:ERA} with the initial condition $z_0\in V_f(\rho)$, and $\epsilon>0$. Suppose that Assumption \ref{assum:quad:and:convergence} is satisfied and that $j_{out} \geq D$, where
\begin{equation*}
    D \doteq \ceil{5+ \frac{1}{\ln\,15 } \ln\left( 1+ \frac{f(z_0)-f^*}{\epsilon} \right) }.
\end{equation*}
Then, 
$m_{\ell+1} \leq \fracg{1}{\sqrt{15}}m_{\ell+1+D},\; \forall \ell \in  \Z_0^{j_{out}-D}.$
\end{lemma}

\begin{proof}
    The proof is obtained by reductio ad absurdum. If there is $\ell \in  \Z_0^{j_{out}-D}$ such that $m_{\ell+1} > \frac{1}{\sqrt{15}}m_{\ell+1+D}$, then we obtain from Lemma \ref{lemma:T}.\ref{prop:T:T} (see Appendix \ref{app:technical:lemmas}) that 
\begin{equation*}
    D < 5+ \frac{1}{\ln\,15 } \ln\left( 1+ \frac{f(z_0)-f^*}{\epsilon} \right),
\end{equation*}
which contradicts the definition of $D$. \qedhere
\end{proof}

\begin{proof}[Proof of Theorem \ref{theo:conv:optimal:algorithm}]

Let $T \in \Z$ be such that
\begin{equation}\label{equ:dj:epsilon:d}
    f(z_{j})-f(z_{j+1}) > \epsilon, \; \forall j \in \Z_0^T,
\end{equation}
is satisfied. Then, defining $d_j \doteq f(z_{j})-f(z_{j+1})$, we have
\begin{equation*}
    f(z_0)-f(z_{T+1}) = \Sum{j=0}{T} d_j \geq (T+1)\left(\min\limits_{j=0,\ldots,T} d_j\right) > (T+1)\epsilon.
\end{equation*}
Thus,
    $T+1 < \fracg{f(z_0)-f(z_{T+1})}{\epsilon} \leq \fracg{f(z_0)-f^*}{\epsilon}\leq \fracg{\rho}{\epsilon}$,
from where we infer that the largest integer $T$ satisfying \eqref{equ:dj:epsilon:d} is bounded.
Consequently, the exit condition of Algorithm \ref{alg:Opt:Alg:ERA} (see step \ref{line:exit:condition}) is satisfied within a finite number of iterations, thus proving claim \ref{theo:conv:optimal:algorithm:finite}.

To prove claim \ref{theo:conv:optimal:algorithm:iter:Ad}, we start by noting that both $m_1$ and $m_2$ are no larger than $\ceil{4 \bn_\rho}$. Indeed, from step \ref{line:sj} we have that $s_0 = s_1 = 0$, which, in virtue of step \ref{line:nj}, implies that $n_0 = m_0 = 1$ and $n_1 = m_1$.
Since $n_0=1$ is no larger than $\ceil{4 \bn_\rho}$ we have from \eqref{ineq:m:jplut:nj} that $m_1$ is also upper-bounded by $\ceil{4 \bn_\rho}$. Moreover, since $n_1=m_1\leq \ceil{4 \bn_\rho}$, we obtain by the same reasoning that $m_2\leq \ceil{4 \bn_\rho}$.
We now prove that if $j \geq 2$ and $m_j \leq \ceil{4 \bn_\rho}$, then $m_{j+1} \leq \ceil{4 \bn_\rho}$. From step \ref{line:sj} we have
\begin{align*}
s_j^2 & = \frac{f(z_{j-1})-f(z_{j})}{f(z_{j-2})-f(z_j)} =   1-\frac{f(z_{j-2})-f(z_{j-1})}{f(z_{j-2})-f(z_j)} \\
      &\leq   1-\frac{f(z_{j-2})-f(z_{j-1})}{f(z_{j-2})-f^*}  \\
      &= \frac{f(z_{j-1})-f^*}{f(z_{j-2})-f^*} \becauseof[\leq]{\eqref{ineq:recursion:improvement}}  \left(\frac{\bn_\rho}{m_{j-1}+1}\right)^2.
\end{align*}
Thus, we have $s_j m_{j-1} \leq \bn_{\rho}$. Therefore,
\begin{equation*}
n_j = \max\{ m_j, 4 s_j m_{j-1} \} \leq \max\{ \ceil{4 \bn_\rho}, 4 \bn_\rho  \} = \ceil{4 \bn_\rho},
\end{equation*}
which, along with \eqref{ineq:m:jplut:nj}, leads to $m_{j+1}\leq  \ceil{4\bn_\rho}$, thus proving the claim.

Finally, to prove claim \ref{theo:conv:optimal:algorithm:iter:total}, we start by noting that the computation of each $z_{j+1}$ is obtained from $m_{j+1}$ iterations of $\cA$. Thus,
\begin{equation}\label{ineq:Naive:bound:on:NA}
N_\cA= \Sum{j=0}{j_{out}} m_{j+1} \becauseof[\leq]{\eqref{ineq:bound:on:mplus}} (1+j_{out})\ceil{4\bn_\rho}.
\end{equation}
Let us denote
\begin{equation*}
D \doteq \left\lceil 5 + \frac{1}{\ln 15  } \ln\left(1+ \frac{f(z_0)-f^*}{\epsilon} \right) \right\rceil.
\end{equation*}
Consider first the case $j_{out}<D$. Since both $j_{out}$ and $D$ are integers we infer from this inequality that $1+j_{out}\leq D$. 
This, along with \eqref{ineq:Naive:bound:on:NA}, implies that $N_\cA\leq \ceil{4\bn_\rho}D\leq \bar{N}_\cA$.

Suppose now that $j_{out}\geq D$. We first recall that Property \ref{prop:optimal:algorithm}.\ref{prop:optimal:algorithm:mj} states that the sequence $\{m_{j+1}\}_{j\geq 0}$ is non-decreasing. We now rewrite $j_{out}$ as $j_{out}=d+tD$, where $d\in \N_{0,D-1}$ and $t$ is a non-negative integer. Thus,
\begin{align*}
    N_\cA &= \Sum{j=0}{j_{out}} m_{j+1} = \Sum{j=0}{d} m_{j+1} + \Sum{j=1}{tD} m_{d+j+1}\\
          &\leq  D m_{d+1} +D \Sum{i=1}{t}m_{d+1+iD} = D\Sum{i=0}{t}m_{d+1+iD}.
\end{align*}
From Lemma \ref{lemma:D} we have
\begin{equation*}
m_{d+1+iD} \leq \frac{m_{d+1+(i+1)D}}{\sqrt{15}},\; \forall i \in\Z_0^{t-1}.
\end{equation*}
Thus,
$N_\cA \leq D\Sum{i=0}{t} m_{d+1+tD} \left( \fracg{1}{\sqrt{15}} \right)^{t-i}$.

\noindent Using now $m_{1+d+tD} \leq \bm=\ceil{4\bn_\rho}$ (see \eqref{ineq:bound:on:mplus}) we obtain
\begin{align*}
    \frac{N_\cA}{D\bm} & \leq \Sum{i=0}{t} \left( \frac{1}{\sqrt{15}} \right)^{t-i} = \Sum{j=0}{t}\left( \frac{1}{\sqrt{15}} \right)^j \\
    &\leq \Sum{j=0}{\infty}\left( \frac{1}{\sqrt{15}} \right)^j = \frac{\sqrt{15}}{\sqrt{15}-1}\leq \frac{e}{2}.
\end{align*}
Thus, $N_\cA \leq   \fracg{e}{2}\bm D =
\fracg{e}{2} \ceil{4 \bn_\rho}D$. \qedhere

\end{proof}

\subsection{Technical results on the iterates of Algorithm \ref{alg:Opt:Alg:ERA}} \label{app:technical:lemmas}

\begin{lemma}\label{lemma:lower:bound:expression}
The function $\varphi(s): \R \rightarrow \R$, defined as
\begin{equation*}
    \varphi(s) \doteq \left(\frac{1}{s^2}-1\right)\cdot\max\left\{1,(4s)^4\right\},
\end{equation*}
satisfies $\varphi(s) \geq 15, \; \forall s \in(0,\fracg{\sqrt{15}}{4}]$.
\end{lemma}

\begin{proof}
We have that
\begin{equation*}
    \varphi(s) = \bsis{lcl} 
    4^4(s^2-s^4) & \mbox{if} & s>\frac{1}{4}, \\
    \fracg{1}{s^2}-1 & \mbox{if} & s\leq \frac{1}{4}.\esis
\end{equation*}
It is clear that $\varphi(\cdot)$ is monotonically decreasing in $(0,\frac{1}{4}]$. Thus, 
$$ \min\limits_{s\in(0,\frac{\sqrt{15}}{4}]} \varphi(s) =  \min\limits_{s\in[\frac{1}{4},\frac{\sqrt{15}}{4}]} \varphi(s) =
 \min\limits_{s\in[\frac{1}{4},\frac{\sqrt{15}}{4}]}  4^4(s^2-s^4) .$$
We notice that the derivative of $s^2-s^4$ is $2s(1-2s^2)$, which vanishes only once in the interval of interest (at $s=\frac{1}{\sqrt{2}}$). From here we infer that $s^2-s^4$ is increasing in $[\frac{1}{4},\frac{1}{\sqrt{2}})$ and decreasing in $(\frac{1}{\sqrt{2}}, \frac{\sqrt{15}}{4}]$. Thus, the minimum is attained at the extremes of the interval $[\frac{1}{4},\frac{\sqrt{15}}{4}]$. That is, we conclude that
\begin{equation*}
    \min\limits_{s\in(0,\frac{\sqrt{15}}{4}]} \varphi(s) = \min\{ \varphi(\frac{1}{4}), \varphi(\frac{\sqrt{15}}{4})\} = \min\{ 15,15\}=15. \qedhere
\end{equation*}
\end{proof}

\begin{lemma}[Technical results on the iterates of Alg. \ref{alg:Opt:Alg:ERA}] \label{lemma:T}
Consider Algorithm \ref{alg:Opt:Alg:ERA} with the initial condition $z_0\in V_f(\rho)$, and $\epsilon>0$. Suppose that Assumption \ref{assum:quad:and:convergence} is satisfied and that $j_{out}\geq 2$. 
Suppose also that there is $T\in \Z_2^{j_{out}}$ and  $\ell\in \Z_0^{j_{out}-T}$ such that
$m_{\ell+1} > \fracg{1}{\sqrt{15}}m_{\ell+1+T}.$
Then: 
\blista
\item $s_j\in \left(0, \frac{\sqrt{15}}{4}\right]$, $\forall j\in \Z_{\ell+2}^{\ell+T}$.
\item $\Sum{j = \ell+2}{\ell+T} \ln\left( \max\,\{1,(4s_j)^4\}\right)< 4\ln\, 15$.
\item $\Sum{j = \ell+2}{\ell+T} \ln \, (\fracg{1}{s_j^2}-1)\leq \ln\left( 1+\fracg{f(z_0)-f^*}{\epsilon}\right)$.
\item $T<5+\fracg{1}{\ln\,15} \ln\left( 1+\fracg{f(z_0)-f^*}{\epsilon}\right)$. \label{prop:T:T}
\elista
\end{lemma}

\begin{proof}

Denote $f_j=  f(z_j)$, $j\in \Z_0^{j_{out}+1}$. From $j\geq 2$ and step \ref{line:sj} of Algorithm \ref{alg:Opt:Alg:ERA} we have 
$$s_j^2=\frac{f_{j-1}-f_{j}}{f_{j-2}-f_{j}}, \; j\in \Z_2^{j_{out}}.$$
The inequality $s_j>0$, $\forall j\in \Z_{\ell+2}^{\ell+T}$ follows from Property \ref{prop:optimal:algorithm}.\ref{prop:optimal:algorithm:sj}.
In order to prove the first claim it remains to prove the inequality $s_j \leq \frac{\sqrt{15}}{4}$, $\forall j\in \Z_{\ell+2}^{\ell+T}$. We proceed by reductio ad absurdum. Suppose that there is $j\in \Z_{\ell+2}^{\ell+T}$ such that $s_j>\frac{\sqrt{15}}{4}$. In this case,  
$$m_{j+1} \becauseof[\geq]{\eqref{ineq:m:n:m}} n_j =\max\{ m_{j},  4s_j m_{j-1} \} \geq 4 s_j m_{j-1} > \sqrt{15} m_{j-1},$$
which along the non-decreasing nature of the sequence $\{m_j\}$  (Property \ref{prop:optimal:algorithm}.\ref{prop:optimal:algorithm:mj}) leads to
\begin{equation*}
m_{\ell+1+T}\geq m_{j+1} > \sqrt{15}m_{j-1} \geq  \sqrt{15} m_{\ell+1},
\end{equation*}
contradicting the assumption of the property, proving claim~\textit{(i)}.

From the non-decreasing nature of the sequence $\{m_j\}$ (Property \ref{prop:optimal:algorithm}.\ref{prop:optimal:algorithm:mj}) we have, for every $j\in\Z_{\ell+2}^{\ell+T}$,
\begin{equation*}
    m_{j+1} \becauseof[\geq]{\eqref{ineq:m:n:m}} n_j =\max \,\{ m_{j},  4s_j m_{j-1} \}\geq  m_{j-1} \cdot \max \,\{ 1,  4s_j \}.
\end{equation*}
Equivalently,
$\ln \left( \max\,\left\{ 1, 4s_j \right\}\right)  \leq \ln \frac{m_{j+1}}{m_{j-1}},  \; \forall j \in \Z_{\ell+2}^{\ell+T}$.
This implies
\begin{align} 
    & \Sum{j=\ell+2}{\ell+T}\ln \left( \max\,\left\{ 1, 4s_j \right\}\right)  \leq   \Sum{j=\ell+2}{\ell+T}\ln \frac{m_{j+1}}{m_{j-1}}  \nonumber \\
    & =  \ln \frac{m_{\ell+T} m_{\ell+1+T}}{m_{\ell+1}m_{\ell+2}}  \leq  \ln \frac{ m_{\ell+1+T}^2}{m_{\ell+1}^2} =2 \ln \frac{ m_{\ell+1+T}}{m_{\ell+1}} \nonumber \\
    & < 2 \ln\, \sqrt{15} = \ln 15. \label{ineq:maxi:ln}
\end{align}
The second claim is obtained multiplying the last inequality by 4. To prove the third claim we notice that 
\begin{equation*}
\Prod{j=\ell+2}{\ell+T} (\fracg{1}{s_j^2} -1) = \Prod{j=\ell+2}{\ell+T} \fracg{f_{j-2}-f_{j-1}}{f_{j-1}-f_j} = \frac{f_{\ell}-f_{\ell+1}}{f_{\ell+T-1}-f_{\ell+T}}.
\end{equation*} 
Since $\ell+T\leq j_{out}$ we have $f_{\ell+T-1}-f_{\ell+T} > \epsilon >0$. Using this inequality we obtain
\begin{equation*}
\Prod{j=\ell+2}{\ell+T}  (\fracg{1}{s_j^2} -1) < \frac{f_{\ell}-f_{\ell+1}}{\epsilon} \becauseof[\leq]{\eqref{ineq:recursion:improvement}} \frac{f_{0}-f_{\ell+1}}{\epsilon} \leq  \frac{f_0-f^*}{\epsilon},
\end{equation*} 
from where the third claim directly follows.
In order to prove the last claim of the property we sum the inequalities given by the second and third claims to obtain
\begin{align} \label{ineq:ln:E}
    &\Sum{j=\ell+2}{\ell+T} \ln \left(\left(\fracg{1}{s_j^2}-1\right)\cdot\max\left\{1,(4s_j)^4\right\}  \right) \nonumber \\
                          &< \ln\left( 1+\fracg{f_0-f^*}{\epsilon} \right)+ 4\ln 15.
\end{align}
From the first claim we have $s_j \in \left(0, \frac{\sqrt{15}}{4}\right]$, $\forall j\in  \Z_{\ell+2}^{\ell+T}$.
Thus, the left term of \eqref{ineq:ln:E} can be lower bounded by means of the following inequality (Lemma \ref{lemma:lower:bound:expression})
\begin{equation*}
    15 \leq \left(\frac{1}{s^2}-1\right)\cdot\max\left\{1,(4s)^4\right\} , \; \forall s \in \left( 0, \textstyle{\frac{\sqrt{15}}{4}} \right].
\end{equation*}
That is,  
$\Sum{j=\ell+2}{\ell+T} \ln\,15  < \ln\left(1+ \fracg{f_0-f^*}{\epsilon} \right) + 4\ln\,15.\quad\quad$
Equivalently, $(T {-} 1)  \ln\,15 < \ln\left(1{+} \fracg{f_0-f^*}{\epsilon} \right) + 4\ln\,15. \qedhere$
\end{proof}
    
\end{appendix}

\bibliographystyle{IEEEtran}
\bibliography{IEEEabrv, Bib_Restart_First_Order_Methods}

\end{document}